\documentclass{article}

\usepackage[letterpaper,margin=1in]{geometry}
\usepackage{mathtools, verbatim, rotating, comment, epic, amsmath, amsfonts, amsthm, amssymb, graphicx, tablefootnote}
\usepackage[dvipsnames]{xcolor}
\usepackage[shortlabels]{enumitem}
\usepackage{url,hyperref,multirow}

\newtheorem{thm}{Theorem}
\newtheorem{lem}{Lemma}

\newtheorem{conj}{Conjecture}

\theoremstyle{definition}

\theoremstyle{remark}
\newtheorem{rem}{Remark}

\newcommand{\NN}{\mathbb{N}}

\newcommand{\ZZ}{\mathbb{Z}}

\newcommand{\D}{\mathcal{D}}

\renewcommand{\S}{\mathcal{S}}

\newcommand{\bs}{\backslash}

\renewcommand{\to}{\mathop{\rightarrow}\limits}

\newcommand{\Size}[1]{\left\lvert #1 \right\rvert}

\newcommand{\union}{\cup}

\newcommand{\<}{\left\langle}
\renewcommand{\>}{\right\rangle}
\renewcommand{\(}{\left(}
\renewcommand{\)}{\right)}

\newcommand{\ignore}[1]{}

\renewcommand{\epsilon}{\varepsilon}

\begin{document}

\title{Degree asymptotics of the numerical semigroup tree}

\author{Evan O’Dorney}

\maketitle

\noindent Semigroup Forum (2013) \textbf{87}:601–616;
DOI 10.1007/s00233-013-9486-7

\noindent Received: 21 October 2012 / Accepted: 1 March 2013 / Published online: 9 April 2013

\noindent \emph{\small Published version © Springer Science+Business Media New York 2013. This version (2024) corrects some typos, especially in the statement and proof of Lemma \ref{lem:11}.}

\begin{abstract}
A \emph{numerical semigroup} is a subset $\Lambda$ of the nonnegative integers that is closed under addition, contains $0$, and omits only finitely many nonnegative integers (called the \emph{gaps} of $\Lambda$). The collection of all numerical semigroups may be visually represented by a tree of element removals, in which the children of a semigroup $\Lambda$ are formed by removing one element of $\Lambda$ that exceeds all existing gaps of $\Lambda$. In general, a semigroup may have many children or none at all, making it difficult to understand the number of semigroups at a given depth on the tree. We investigate the problem of estimating the number of semigroups at depth $g$ (i.e.\ of genus $g$) with $h$ children, showing that as $g$ becomes large, it tends to a proportion $\phi^{-h-2}$ of all numerical semigroups, where $\phi$ is the golden ratio.
\end{abstract}

\paragraph{Keywords} Numerical semigroup · Semigroup tree · Genus · Efficacy

\section{Introduction} \label{sec:1}

A \emph{numerical semigroup} is a subset $ \Lambda $ of the nonnegative integers $\NN_0$ that is closed
under addition, contains $ 0 $, and contains all but finitely many nonnegative integers.
Numerical semigroups have notable applications to algebraic geometry (see \cite{5} for
more details), but the decade of the 2000's has seen a surge of purely combinatorial interest in
them, and in particular the first book-length treatment of them \cite{3} has been published.

Various researchers in the field have investigated counting numerical semigroups
by genus, multiplicity, and/or Frobenius number. The \emph{genus} $ g = g(\Lambda) $ of a numerical
semigroup is the size of the set of gaps, $ g = |\NN_0\bs \Lambda| $. The \emph{multiplicity} is the smallest nonzero element, $ m = m(\Lambda) = \min(\Lambda\bs\{0\}) $, while the \emph{Frobenius number} is the
largest nonnegative integer not included, $ f = f (\Lambda) = \max(\NN_0\bs \Lambda) $.

Research in the field of numerical semigroups has focused extensively on the number $ N(g) $ of numerical semigroups of genus $ g $. Building on work by Zhao \cite{8}, Zhai
\cite{7} proved that $ N(g) $ has a Fibonacci-like order of growth:
\begin{thm}[\cite{7}; conjectured by Bras-Amorós in \cite{1}] \label{thm:1} There is a constant $ S > 3.78 $
such that
\[
  N(g) = S\phi^g + o(\phi^g )
\]
where $\phi = (1+\sqrt5)/2$ is the golden ratio.
\end{thm}
Surprisingly enough, it is still not known whether the sequence $\{N(g)\}$ is increasing.
\begin{conj}[\cite{4}] \label{conj:1}
For all $ g \geq 1 $, the inequality $ N(g) < N(g + 1) $ holds.
\end{conj}
Theorem 1 shows that $ \{N(g) \} $ is increasing for sufficiently large $g$, but the bounds
that it gives for when this happens are very large. The even stronger conjecture that
$ N(g + 2) \geq N(g + 1) + N(g) $ has been proposed \cite{1}.

Traditionally, numerical semigroups are written in terms of generating sets. If
$x_1, x_2, \ldots, x_n$ are setwise relatively prime positive integers, then the semigroup generated by $x_1, \ldots, x_n$ is the set of their nonnegative linear combinations:
\[
  \<x_1, x_2, \ldots, x_n\> = \{a_1x_1 + \cdots + a_n x_n \mid a_1, \ldots, a_n \in \NN_0\}.
\]
However, as an alternative to building a semigroup upwards from a finite set, we
can build downwards by removing elements one at a time from the maximal semigroup $\NN_0$. This is the guiding principle of the \emph{semigroup tree,} a useful tool for visually understanding the set of all numerical semigroups. This tool can be dated back
at least to Bras-Amorós \cite{2}; it has also figured prominently in later work including
the proof of Theorem 1. Formally, we assign to any numerical semigroup $ \Lambda \neq \NN_0 $ its
parent $ P (\Lambda) = \Lambda \union f (\Lambda) $. Conversely, if $ \Lambda $ is a numerical semigroup, its children
are the sets of the form $ \Lambda\bs\{x\} $ such that $ x > f (\Lambda) $ and $ \Lambda\bs\{x\} $ is, in fact, a numerical
semigroup. It is not hard to prove that the infinite tree generated by this construction
contains every numerical semigroup exactly once. A portion of it is shown in Figure \ref{fig:1}.
\begin{figure}
  \centering
  \includegraphics[width=0.6\textwidth]{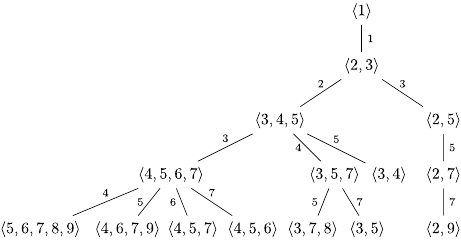}
  \label{fig:1}
  \caption{Tree showing the numerical semigroups of genus $g \leq 4$}
\end{figure}

The \emph{efficacy} $ h(\Lambda) $ of a numerical semigroup $ \Lambda $ is the number of children it has
on the tree. Equivalently, it is the number of positive integers $ x > f (\Lambda) $ such that
$ \Lambda\bs\{x\} $ is still closed under addition. How many such $ x $ exist can vary wildly; in the
tree above, we see that the semigroup $ \<3, 4\> $ is a \emph{leaf}---a node with efficacy zero---%
while the node $ \<4, 5, 6, 7\> $ has efficacy as high as four. This serves to illustrate why
the relations between $ N(g) $ and $ N(g + 1) $ are so complicated and statements such as
Conjecture \ref{conj:1} are difficult to prove. Currently the best bound we have that holds for
all $ g $ is the inequality $ N(g + 1) \geq N(g) - N(g - 1) $, proved by Ye in \cite{6} using the
concept of semigroup efficacy.

Motivated by these concerns, Ye investigated in the same work the question of
estimating the number $ t (g, h) $ of semigroups of genus $ g $ and efficacy $ h $. (In more
graphic terms, $ t (g, h) $ is the number of nodes on the tree at depth $ g $ with $ h $ children.)
Ye proved that there exist constants $c_1$ and $c_2$ such that $c_1\phi^{g-h} \leq t(g, h) \leq c_2\phi^{g-h}$;
she conjectured that, in the limit, these constants are equal and tied to the constant $S$
in Zhai’s estimate for $ N(g) $:

\begin{conj}[Conjecture 2 in \cite{6}] \label{conj:2} For all $h \geq 0$,
\[
\lim_{g\to\infty}
t (g, h)
N(g) = \frac{1}{\phi^{h+2}}.
\]
\end{conj}
In what follows, we will prove this conjecture. We restate it in a strengthened
form emphasizing that $ t (g, h) $ tends asymptotically to its conjectured value, not just
for fixed $ h $, but uniformly.
\begin{thm} \label{thm:2}
Let
\[
S = \lim_{g\to\infty}
\frac{N(g)}{\phi^{g}}
\]
be the semigroup constant from Theorem 1. Then
\begin{equation} \label{eq:1}
\sum_{
h\geq0}
\Size{t (g, h) - S\phi^{g-h-2}} = o(\phi^g ).
\end{equation}
\end{thm}
It is easy to see that, by considering just the term of the left-hand side of \eqref{eq:1}
corresponding to one fixed value of $ h $, Conjecture \ref{conj:2} is a corollary of this.
The remainder of this paper will be a proof of this theorem. Section 2 lays down
various facts that will be basic to our proof. In Sect.\ 3, we will introduce some useful
recursions involving $ t (g, h) $ and another quantity $ s(g, h) $, which counts semigroups
that have in a certain sense “one too many” children. In Sect.\ 4, we will split $ t (g, h) $
into two parts and bound them separately. Finally, in Sect.\ 5, we will reassemble our
bounds to prove Theorem 2.

\section{Preliminary facts} \label{sec:2}

If $ \Lambda $ is a numerical semigroup, then an integer $x$ is called an \emph{effective generator} of $ \Lambda $
if it is the Frobenius number of one of $ \Lambda $’s children; that is, if $ \Lambda\bs\{x\} $ is a numerical
semigroup with parent $ \Lambda $. It is known that a positive integer $ x $ is an effective generator
if and only if $ x > f (\Lambda) $ and $ x $ is not the sum of two nonzero elements of $ \Lambda $. More
pertinent to our argument is the distinction of three kinds of effective generators (see
\cite{7}):
\begin{enumerate}
  \item The effective generators of $ \Lambda $’s parent $ P (\Lambda) $ that exceed $ f (\Lambda) $. We call such numbers \emph{weak generators;} they are always effective generators of $ \Lambda $ as well.
  \item The number $ m(\Lambda) + f (\Lambda) $, which may or may not be an effective generator of $ \Lambda $.
  If it is, we call it a \emph{strong generator} and say that $ \Lambda $ is \emph{strongly descended;} otherwise $ \Lambda $ is \emph{weakly descended.}
  \item The \emph{super-strong generator} $ m(\Lambda) + f (\Lambda) + 1 $, which appears only for the special class of semigroups of the form $ \{0, m, m + 1, m + 2, \ldots\} $ $ (m \geq 2) $. These
  semigroups are also strongly descended, and we call them \emph{super-strongly descended.}
\end{enumerate}
For convenience later, we will declare that the trivial semigroup $ \NN_0 $ is strongly descended but not super-strongly descended.

\subsection{Useful lemmas from Zhai} \label{sec:2.1}
Although Zhai’s paper was primarily aimed at bounding $ N(g) $, his employment of
the semigroup tree led him to prove several lemmas directly relevant to our study
of $ t (g, h) $. Three especially pertinent lemmas were already isolated by Ye in \cite{6}; we
reproduce them here with minor changes in notation.

Let $ s(g, h) $ be the number of \emph{strongly} descended semigroups with genus $ g $ and
efficacy $ h $. Since, in the next section, we will express $ t (g, h) $ as a simple expression in
$ s(g, h) $, it is useful to have accurate knowledge of the behavior of $ s(g, h) $.
The following lemma states that in a certain domain, $ s(g, h) $ depends only upon the
difference $ g - h $.
\begin{lem}[Lemma 7 from \cite{7}] \label{lem:1} Let $ r(n) = s(2n + 1, n + 1) $. If $ h < g < 2h $, then
$ s(g, h) = r(g - h) $.
\end{lem}
When $ h \geq g $, it is easy to see that
\[
  s(g,h) = \begin{cases}
    1 & \text{if }h = g + 1 \\
    0 & \text{otherwise}
  \end{cases}
\]
since the only semigroups satisfying this
inequality are the super-strongly descended ones with $ h = g + 1 $. Computing $ s(g, h) $
when $ g \geq 2h $ is much more challenging. This lemma therefore partitions the domain
$ \NN_0 \times \NN_0 $ of allowable ordered pairs $ (g, h) $ into a “good” region where $ g < 2h $ and a
“bad” region where $ g \geq 2h $. We will generally consider the values of $ s(g, h) $ in these
two regions separately.

The second lemma that we will use is a bound on the rate of growth of $ s(g, h) $ in
the “good” region, as signaled by the function $ r(n) $:
\begin{lem}[Lemma 8 from \cite{7}] \label{lem:2} The sum $ \sum_{
n\geq1} \phi^{-n} r(n) $ converges.
\end{lem}

The third and final lemma is a bound on the size of $ s(g, h) $ in the “bad” region. It
is not specifically numbered in Zhai’s work, but it will be crucial to our proof.

\begin{lem}[from Sect.\ 3.2, page 10 of \cite{7}] \label{lem:3} For $a > 0$ an integer, let $ \D_a $ be the set
of strongly descended semigroups $ \Lambda $ satisfying the following inequalities:
\begin{align}
g(\Lambda) + h(\Lambda) \geq a \label{eq:2} \\
g(\Lambda) - h(\Lambda) \geq \frac{a}{3} \label{eq:3} \\
g(\Lambda) \leq a. \label{eq:4}
\end{align}
Then
\[
\lim_{a\to\infty}
\sum_{
\Lambda\in \D_a}
\phi^{h(\Lambda)-g(\Lambda)} = 0.
\]
\end{lem}
\begin{proof}
For ease of comparison, we remark that our $ a $ corresponds to Zhai’s $ g $, while
Zhai defines $\S_3$ to be the set of numerical semigroups satisfying \eqref{eq:2} and \eqref{eq:3} above.
He replaces Eq.\ \eqref{eq:4} with the weaker condition that the multiplicity (minimal nonzero
element) of $ \Lambda $ is at least $ a + 1 $. Finally, the sums are multiplied by $ \phi^{g} $ in Zhai’s
computations. \end{proof}

\subsection{Fibonacci numbers} \label{sec:2.2}
We conclude this preliminary section by briefly laying out some basic facts about
Fibonacci numbers. We use the indexing convention that $F_0 = 0$ and $F_1 = 1$; the
remainder of the sequence is constructed by the recursion $F_n = F_{n-1} + F_{n-2}$. We
will make use of Binet’s formula
\begin{equation} \label{eq:5}
  F_n = \frac{1}{\sqrt5}\left[\phi^n + \(-\frac{1}{\phi}\)^n\right]
\end{equation}
as well as the well-known connection to Pascal’s triangle:
\begin{equation} \label{eq:6}
\sum_{
k\in \ZZ}
\binom{k}{n - k} = F_{n+1}.
\end{equation}
(Here we make the convention that $\binom{a}{b} = 0$ if $ b < 0 $ or $ b > a $.) Also, we will use the
following crude upper bound:
\begin{lem} \label{lem:4} For $ n \geq 0 $, we have $ F_n \leq \phi^{n-1}$. \end{lem}
\begin{proof}Clearly, $ F_0 = 0 < \phi^{-1} $ and $F_1 = 1 = \phi^0$. Since $\phi^n = \phi^{n-1} + \phi^{n-2}$, the result
follows by induction. \end{proof}

\section{Recursions} \label{sec:3}
Our goal in this section is to reduce the computation of $ t (g, h) $ to that of more easily estimable quantities. We will express $ t (g, h) $ in terms of $ s(g, h) $ (the number of
strongly descended semigroups of genus $ g $ and efficacy $ h $), first recursively and then
explicitly; although $ s(g, h) $ and $ t (g, h) $ are both difficult to compute exactly, $  s(g, h) $
is of smaller magnitude and accordingly easier to bound.

As before, let $ s(g, h) $ be the number of strongly descended semigroups with genus
$ g $ and efficacy $ h $. Also, define
\[
\hat s(g, h) = \begin{cases}
  1 & \text{if $g = h$ (in contrast to $s(g, h) = 0$)} \\
  s(g, h) & \text{otherwise.}
\end{cases}
\]
This is a technical correction to compensate for the existence of the super-strong
generators, as will be made clear later. We first prove a recursion that,
combined with the base cases $ t (g, h) = 0 $ for $ h \geq g + 2 $, derives $ t (g, h) $ from $ \hat s(g, h) $.

\begin{lem} \label{lem:5} For all $ g, h \geq 0 $ except $ g = h = 0 $,
\begin{equation} \label{eq:7}
  t (g, h) = t (g, h + 1) + t (g - 1, h + 1) + \hat s(g, h) - 2\hat s(g, h + 1) + \hat s(g, h + 2).
\end{equation}
\end{lem}

\begin{rem}
There is an anomaly at $ g = h = 0 $ due to the nonexistence of a super-strongly
generated semigroup of genus 0. Since $ t (0, 0) $ does not appear in the computation of
any other $ t (g, h) $ values, this is not a concern for us. However, note that \eqref{eq:7} correctly
derives the values $ t (0, 1) = 1 $ and $ t (0, h) = 0 $ for $ h \geq 2 $.
\end{rem}

\begin{proof}In view of the preceding remark, we may assume that $ g \geq 1 $.
  
We begin by computing $ t_{\mathrm w}(g, h) $, the number of genus-$ g $ semigroups having $ h $
weak generators. If a semigroup $ \Lambda $ of genus $ g - 1 $ has $ k $ effective generators, then
its $k$ children will respectively have $ k - 1, k - 2, \ldots, 1, 0 $ weak generators. Thus, if
we are only concerned with children having a fixed number $ h $ of weak generators,
there will be one for every $ \Lambda $ of genus $ g - 1 $ having more than $ h $ effective generators:
\[
  t_{\mathrm w}(g, h) = \sum_{
k>h}
t (g - 1, k).
\]

Now we would like to transform $ t_{\mathrm w}(g, h) $ into $ t (g, h) $. Assume first that $ h \leq g - 2 $,
so that the super-strong generators do not come into play. There are $ s(g, h) $ strongly
generated semigroups with $ h $ effective generators, $ h - 1 $ of which are weak, that
are counted by $ t (g, h) $ but not by $ t_{\mathrm w}(g, h) $; on the other hand, there are $ s(g, h + 1) $
strongly generated semigroups with $ h + 1 $ effective generators, $ h $ of which are weak,
that are counted by $ t_{\mathrm w}(g, h) $ but not by $ t (g, h) $. Thus, if $ h \leq g - 2$,
\begin{equation} \label{eq:8}
t (g, h) = t_{\mathrm w}(g, h) + s(g, h) - s(g, h + 1) = \sum_{
k>h}
t (g - 1, k) + s(g, h) - s(g, h + 1).
\end{equation}
If we consider the effect of the single super-strongly generated semigroup $ \{0\} \union
[g + 1, \infty) $ having $ g + 1 $ effective generators, only $ g - 1 $ of which are weak (for
$ g \geq 1 $), we find that the right-hand side of \eqref{eq:8} is one too high if $ g = h - 1 $ and one too
low if $ g = h $. Both of these problems are fixed if the terms $ s $ are changed to $ \hat s $, so
\[
t (g, h) = \sum_{
k>h}
t (g - 1, k) + \hat s(g, h) - \hat s(g, h + 1).
\]

Finally, to eliminate the sum, we change $ h $ to $ h + 1 $ and subtract, getting
\[
t (g, h) - t (g, h + 1) = t (g - 1, h + 1) + \hat s(g, h) - 2\hat s(g, h + 1) + \hat s(g, h + 2).
\]
Moving the term $ t (g, h + 1) $ to the other side gives the desired result. \end{proof}
We will have several occasions to discuss recursions similar in form to \eqref{eq:7}, and
therefore we make the following definition. If $ S(g, h) $ is a function on pairs of nonnegative integers satisfying $S(g, h) = 0$ for $h > g + 1$, then the function $ T_S (g, h) $ is
defined for $ g \geq 1 $ and $ h \geq 0 $ by the recursion
\begin{equation} \label{eq:9}
  T_S (g, h) = T_S (g, h + 1) + T_S (g - 1, h + 1) + S(g, h)
\end{equation}
with the initial conditions that $ T_S (g, h) = 0 $ for all $ g, h \in \ZZ $ such that $ h > g + 1 $. Thus
$ t (g, h) = T_S (g, h) $ for the choice $ S(g, h) = \hat s(g, h) - 2\hat s(g, h + 1) + \hat s(g, h + 2) $. We
would like to work in terms of the simpler choice $ S = \hat s $. We note that the recursion
\eqref{eq:9} is linear, so $ T_{S_1+S_2} = T_{S_1} + T_{S_2} $. Also, it has a translation-invariance property: if
we define $ S'(g, h) = S(g, h + 1) $, then $ T_{S'} = (T_{S} )' $. In particular, for $ (g, h) \neq (0, 0) $,
\begin{equation} \label{eq:10}
\begin{aligned}
t (g, h) &= T_{\hat s}-2\hat s' +\hat s'' (g, h) = T_{\hat s} (g, h) - 2T '
\hat s (g, h) + T ''
\hat s (g, h) \\
&= T_{\hat s} (g, h) - 2T\hat s (g, h + 1) + T_{\hat s} (g, h + 2).
\end{aligned}
\end{equation}
This enables us to reduce questions about $ t $ to questions about $ T_{\hat s} $.

We next write an explicit solution to \eqref{eq:9}.

\begin{lem} \label{lem:6} If $ S \colon \NN_0 \times \NN_0 \to \ZZ $ is a function with $ S(g, h) = 0 $ for $ h > g + 1 $, then
\begin{equation} \label{eq:11}
  T_S (g, h) = \sum_{
    x,y\geq0}
    \binom{y - h}{g - x}
    S(x, y),
\end{equation}
where we make the convention that $\binom{n}{k} = 0$ when $k < 0$ or $k > n$.
\end{lem}
\begin{proof}Although this can easily be proved by induction on $g - h$, using the identity
\[
\binom{n}{k} - \binom{n - 1}{k} - \binom{n - 1}{k - 1} = \begin{cases}
  1, & n = k = 0 \\
  0, & \text{otherwise,}
\end{cases}
\]
we find it illuminating to present the following more intuitive argument.

If $ S(x, y) $ is treated as a separate variable for each pair $ (x, y) $, devoid of numerical
value, then $ T_S (g, h) $ will be a homogeneous linear polynomial in the $ S(x, y) $. We
therefore fix $ x $ and $ y $ and consider the contribution of $ S(x, y) $ to $ T_S (g, h) $. We note that
unless $ g = x $ and $ h = y $, the instances of $ S(x, y) $ on the right-hand side of \eqref{eq:9} cannot
come from $ S(g, h) $ and must come from one of the $ T $ terms. These $ T $ terms pick up
instances of $ S(x, y) $ from their own $ T $ terms in \eqref{eq:9}, and this continues recursively
until we reach $ T_S (x, y) $, which gets a single instance of $ S(x, y) $ from the last term
of \eqref{eq:9}. Thus $ S(x, y) $ appears in $ T_S (g, h) $ once for each path in the coordinate plane
from the point $ (g, h) $ to $ (x, y) $ using the steps $ (0, 1) $ and $ (-1, 1) $. The number of such
paths is the binomial coefficient $\binom{y - h}{g - x}$, and \eqref{eq:11} follows. \end{proof}

\section{Bounding \texorpdfstring{$T_{\hat s}$}{Tŝ}}\label{sec:4}
In Sect.\ 2.1, we defined
\[
r(n) = s(2n + 1, n + 1) = \lim_{
h\to\infty} s(h + n, h) \quad (n > 0).
\]
For convenience, let us extend the definition to all integers $n$ as follows:
\[
r(n) = \lim_{
h\to\infty} \hat s(h + n, h) =
\begin{cases}
s(2n + 1, n + 1) & \text{if }n \geq 1 \\
1 & \text{if } n = 0 \text{ or } n = -1 \\
0 & \text{if } n \leq -2.
\end{cases}
\]
Then it is not hard to check that
\[
\hat s(g, h) = r(g - h)
\]
for all $g, h \geq 0$ satisfying $g < 2h$.

Also, define $s_1(g, h) = r(g - h)$ and $s_2(g, h) = \hat s(g, h) - s_1(g, h)$, noting that
$s_2(g, h) = 0$ for $g < 2h$. We will bound $T_{\hat s} (g, h)$ by means of its components $T_{s_1} (g, h)$
and $T_{s_2} (g, h)$.

\subsection{Bounding \texorpdfstring{$T_{s_2}$}{Ts2}} \label{sec:4.1}
We begin by proving that the $T_{s_2}$ term is negligible for large $g$.
\begin{lem} \label{lem:7} As $g \to \infty$,
\[
\sum_{
h\geq0}
\Size{T_{s_2} (g, h)} = o(\phi^g ).
\]
\end{lem}
\begin{proof} Using Lemma \ref{lem:6}, we have
\begin{align}
\Size{T_{s_2} (g, h)}
&= \Size{
\sum_{
x,y\geq0}
\binom{y - h}{g - x}
s_2(x, y)} \nonumber \\
&=
\Size{\sum_{\substack{
x,y\geq0\\
x\geq2y}}
\binom{y - h}{g - x}\big(\hat s(x, y) - s_1(x, y)\big)} \nonumber \\
&\leq \sum_{\substack{
x,y\geq0\\
x\geq2y}}
\binom{y - h}{g - x}
\hat s(x, y) + \sum_{\substack{
    x,y\geq0\\
    x\geq2y}}
\binom{y - h}{g - x}
s_1(x, y) \nonumber \\
&\leq \sum_{\substack{
    x,y\geq0\\
    x\geq2y}}
\binom{y - h}{g - x}
s(x, y) + \sum_{\substack{
    x,y\geq0\\
    x\geq2y}}
\binom{y - h}{g - x}
s_1(x, y). \label{eq:12}
\end{align}
The result will follow from the next two lemmas, which bound the two sums in \eqref{eq:12}
separately. \end{proof}
\begin{lem} \label{lem:8} Let
\[
  u(g, h) = \sum_{\substack{
x,y\geq0\\
x\geq2y}}
\binom{y - h}{g - x}
s_1(x, y).
\]
Then
\[
  \sum_{
h\geq0}
u(g, h) = o(\phi^g).
\]
\end{lem}
\begin{proof} Recalling the definition of $s_1$ in terms of $r$, we write
\[
u(g, h) = \sum_{\substack{
x,y\geq0\\
x\geq2y}}
\binom{y - h}{g - x}
r(x - y) = \sum_{\substack{
y,n\geq0\\
n\geq y}}
\binom{ y - h}
{g - y - n}
r(n).
\]
We claim that in order for the summand to be nonzero, we must have $ g/3 \leq n \leq g $.
The upper bound $ n \leq g $ is trivial since the lower member $ g - y - n $ of the binomial
coefficient must be nonnegative. To prove that $ n \geq g/3 $, we note (again using the
binomial coefficient) that
\[
  0 \leq (y - h) - (g - y - n) = 2y - h + n - g.
\]
Applying the bounds $ h \geq 0 $ and $ y \leq n $ gives
\[
  0 \leq 2y - h + n - g \leq 2y + n - g \leq 3n - g,
\]
as desired.
Therefore,
\[
u(g, h) \leq \sum_{
g/3 \leq n\leq g}
\sum_{
y\in \ZZ}
\binom{ y - h}
{g - y - n}
r(n).
\]
Using the Fibonacci identity \eqref{eq:6} gives
\begin{align*}
u(g, h) &\leq \sum_{
g/
3 \leq n\leq g}
F_{g-h-n+1}r(n) \\
&\leq \sum_{
g/
3 \leq n\leq g}
\phi^{g-h-n} r(n) \quad \text{(by Lemma \ref{lem:4})} \\
&= \phi^{g-h} \sum_{
g/
3 \leq n\leq g}
\phi^{-n} r(n).
\end{align*}
From this it follows that $\sum_{ g/3 \leq n\leq g } \phi^{-n} r(n)$
tends to $ 0 $ as $ g \to \infty $, and hence
\[
\phi^{-g} \sum_{
h\geq0}
u(g, h) \leq \sum_{
h\geq0}
\phi^{-h} \sum_{
g/
3 \leq n\leq g}
\phi^{-n} r(n)
\]
does as well. \end{proof}

It remains to bound the first sum in \eqref{eq:12}.
\begin{lem} \label{lem:9} Let
\[
v(a, b) = \sum_{\substack{
x,y\geq0
x\geq2y
}}
\binom{y - b}{a - x}
s(x, y).
\]
Then
\begin{equation} \label{eq:13}
\sum_{
b\geq0}
v(a, b) = o(\phi^a ).
\end{equation}
\end{lem}
\begin{proof} First fix $ a \geq 1 $ and $ b \geq 0 $; we will assume that $ a \geq 2b $, since otherwise
$ v(a, b) = 0 $. We recall that $ s(x, y) $ counts the strongly descended semigroups with
genus $ x $ and efficacy $ y $, and we rewrite the sum so as to index over these semigroups:
\begin{equation} \label{eq:14}
v(g, h) = \sum_{\substack{
\Lambda \text{ strongly descended} \\
g(\Lambda)\geq2h(\Lambda)}}
\binom{h(\Lambda) - b}{a - g(\Lambda)}.
\end{equation}

We would now like to apply Lemma \ref{lem:3}. We claim that all $ \Lambda $ making a nonzero
contribution to \eqref{eq:14} belong to $ \D_a $, i.e. satisfy
\begin{align}
g(\Lambda) + h(\Lambda) &\geq a \label{eq:15} \\
g(\Lambda) - h(\Lambda) &\geq \frac{a}{3} \label{eq:16} \\
g(\Lambda) &\leq a. \label{eq:17}
\end{align}
For the binomial coefficient to be nonzero, we must have
\begin{align*}
h(\Lambda) - b &\geq a - g(\Lambda) \\
g(\Lambda) + h(\Lambda) &\geq a + b \geq a.
\end{align*}
This proves \eqref{eq:15}. To prove \eqref{eq:16}, we use \eqref{eq:15} together with $g(\Lambda) \geq 2h(\Lambda)$ to get
\begin{align*}
g(\Lambda) - h(\Lambda) &= \frac{3g(\Lambda) - 3h(\Lambda)}
{3}\\
&= \frac{(g(\Lambda) + h(\Lambda)) + 2(g(\Lambda) - 2h(\Lambda))}
{3} \geq \frac{a + 0}
{3} = \frac{a}
{3}.
\end{align*}

Finally, we get \eqref{eq:17} from the binomial
coefficient in \eqref{eq:14}.
Using the crude bound
\[
\binom{x}{y}
\leq F_{x+y+1} \leq \phi^{x+y},
\]
we get
\begin{align*}
v(a, b) &\leq \sum_{
\Lambda\in\D_a}
\binom{h(\Lambda) - b}{a - g(\Lambda)}\\
&\leq \sum_{
\Lambda\in\D_a}
\phi^{h(\Lambda)-b+a-g(\Lambda)} \\
&= \phi^{a-b} \sum_{\Lambda\in\D_a} \phi^{h(\Lambda)-g(\Lambda)}.
\end{align*}
The last sum matches the statement of Lemma \ref{lem:3} and is thus $ o(1) $. Moreover, this sum
is independent of $ b $, so we can sum over $ b $ (using the fact that $ \sum_{
b\geq0} \phi^{-b} $ converges)
and derive \eqref{eq:13}. \end{proof}

\subsection{Bounding \texorpdfstring{$T_{s_1}$}{Ts1}} \label{sec:4.2}
Having shown that $ T_{s_2} $ is negligible in absolute value, we turn to $ T_{s_1} $. The basic observation is that, by Lemma \ref{lem:6}, $ T_{s_1} (g, h) $ depends only on $ g - h $:
\begin{align*}
T_{s_1} (g, h) &= \sum_{
x,y\geq0}
\binom{y - h}{g - x}
s_1(x, y) \\
&= \sum_{
x,y\geq0}
\binom{y - h}{g - x}
r(x - y) \\
&= \sum_{
k\geq-1}
\sum_{
y\geq0}
\binom{ y - h}{g - y - k}
r(k) \\
&= \sum_{
-1\leq k\leq g-h}
F_{g-h-k+1} r(k) \\
&= w(g - h),
\end{align*}
where
\[
w(n) = \sum_{
-1\leq k\leq n}
F_{n-k+1}r(k).
\]
We state our main bound on $ T_{s_1} $ in terms of $ w $.
\begin{lem} \label{lem:10} There is a constant
\begin{equation} \label{eq:18}
C = \frac{1}{\sqrt5} \sum_{
k\geq-1}
r(k)\phi^{1-k}
\end{equation}
such that, as $n \to \infty$,
\[
  w(n) = C\phi^n + o(\phi^n).
\]
\end{lem}
Note that $C$ is finite by Lemma \ref{lem:2}. This $C$ is closely related to the semigroup
constant $S$:
\begin{lem} \label{lem:11} $C = S\phi^{2}$.
\end{lem}
However, the proof of this fact will not emerge until the end of the proof of Theorem \ref{thm:2}.
\begin{proof}[Proof of Lemma \ref{lem:10}] We estimate
\begin{align*}
\Size{w(n) - C\phi^n}
&=
\Size{\sum_{
-1\leq k\leq n}
r(k)F_{n-k+1} - \frac{1}
{\sqrt5}
\sum_{
k\geq-1}
r(k)\phi^{n-k+1}} \\
&=
\Size{\sum_{
-1\leq k\leq n}
r(k)
\(
F_{n-k+1} - \frac{\phi^{n-k+1}}
{\sqrt 5}
\)
- \frac{1}
{\sqrt5}
\sum_{
k>n}
r(k)\phi^{n-k+1}} \\
&\leq \sum_{
-1\leq k\leq n}
r(k)
\Size{F_{n-k+1} - \frac{\phi^{n-k+1}}
{\sqrt5}
} + \frac{1}
{\sqrt5}
\sum_{
k>n}
r(k)\phi^{n-k+1} \\
&= \sum_{
-1\leq k\leq n}
r(k) \frac{\phi^{-n+k-1}}{
\sqrt5} + \frac{1}
{\sqrt5}
\sum_{
k>n}
r(k)\phi^{n-k+1} \quad \text{(by \eqref{eq:5})} \\
&= \frac{\phi^{-n-1}}{\sqrt5}
\sum_{
-1\leq k\leq n}
r(k)\phi^{k} + \frac{\phi^{n+1}}{\sqrt5}
\sum_{
k>n}
r(k)\phi^{-k}.
\end{align*}
The second of the last two sums is $ o(1) $ by Lemma \ref{lem:2}. To finish the proof, it suffices
to show that
\begin{equation}
  \sum_{
-1\leq k\leq n}
r(k)\phi^{k} = o(\phi^{2n}). \label{eq:19}
\end{equation}
Let $ \epsilon > 0 $ be given. By Lemma \ref{lem:2}, there is an $ n_0 $ such that $r(k) < \epsilon\phi^{k}$ for all $ k \geq n_0 $.
Then for all $ n \geq n_0 $,
\begin{align*}
\sum_{
-1\leq k\leq n}
r(k)\phi^{k} &< \sum_{
-1\leq k\leq n_0}
r(k)\phi^{k} + \epsilon \sum_{
n_0<k\leq n}
\phi^{2k} \\
&< \sum_{
-1\leq k\leq n_0}
r(k)\phi^{k} + \epsilon \frac{\phi^{2n}}
{1 - \phi^{-2}}.
\end{align*}
Since the first term is independent of $ n $, we have for all sufficiently large $ n $
\[
\sum_{
-1\leq k\leq n}
r(k)\phi^{k} < \epsilon\phi^{2n}
\(
1 + \frac{1}{1 - \phi^{-2}}
\)
.
\]
Since $ \epsilon $ was arbitrary, the result \eqref{eq:19} follows. \end{proof}

\section{Proof of Theorem \ref{thm:2}} \label{sec:5}
We will use the following lemma:
\begin{lem} \label{lem:12}
\begin{equation}
\sum_{
h\geq0}
\Size{t (g, h) - C\phi^{g-h-4}} = o\(\phi^{g} \). \label{eq:20}
\end{equation}
\end{lem}
Before proving this, we explain how it implies Lemma \ref{lem:11} and Theorem \ref{thm:2}.
\begin{proof}[Proof of Lemma \ref{lem:11}]
Given estimates for all $ t (g, h) $, we may estimate their sum, $ N(g) $:
\begin{align*}
N(g) &= \sum_{
h\geq0}
t (g, h) \\
& = \sum_{
h\geq0}
C\phi^{g-h} + o\(\phi^{g} \) \\
&= C \cdot \frac{\phi^{g}}{1 - \phi^{-1}} + o\(\phi^{g} \) \\
&= C\phi^{g+2} + o\(\phi^{g} \).
\end{align*}
However, since we know that $ N(g) = S\phi^{g} + o\(\phi^{g} \) $, the result follows. \end{proof}
Substituting $ C = S\phi^{2} $ into Lemma \ref{lem:12} yields Theorem \ref{thm:2}.
\begin{proof} [Proof of Lemma \ref{lem:12}] For fixed $ h $, we have
\begin{align*}
t (g, h) &= T_{\hat s} (g, h) - 2T\hat s (g, h + 1) + T_{\hat s} (g, h + 2) \quad \text{(by \eqref{eq:10})}
\\ &= T_{s_1} (g, h) - 2T_{s_1} (g, h + 1) + T_{s_1} (g, h + 2) + o\(\phi^{g} \) \text{(by Lemma \ref{lem:7})}
\\ &= w(g - h) - 2w(g - h - 1) + w(g - h - 2) + o\(\phi^{g} \)
\\ &= C\phi^{g-h}(1 - 2\phi^{-1} + \phi^{-2}) + o\(\phi^{g} \) \text{(by Lemma \ref{lem:10})}
\\ &= C\phi^{g-h-4} + o\(\phi^{g} \).
\end{align*}
However, the generality of our lemmas allows us to bound $|t (g, h) - C\phi^{g-h-4}|$ not
just for fixed $h$ but summed over all $h$:
\begin{align*}
&\sum_{
h\geq0}
\Size{t (g, h) - C\phi^{g-h-4}}
\\ &= \sum_{
h\geq0}
\Size{T_{\hat s} (g, h) - 2T_{\hat s} (g, h + 1) + T_{\hat s} (g, h + 2) - C(\phi^{g-h} - 2\phi^{g-h-1} + \phi^{g-h-2})} \\
&\leq \sum_{
h\geq0}
\Big[\Size{T_{\hat s} (g, h) - C\phi^{g-h}} + 2\Size{T_{\hat s} (g, h + 1) - C\phi^{g-h-1}} \\
&{} + \Size{T_{\hat s} (g, h + 2) - C\phi^{g-h-2}}\Big] \\
&\leq 4 \sum_{
h\geq0}
\Size{T_{\hat s} (g, h) - C\phi^{g-h}}
\leq 4 \sum_{
h\geq0}
\Size{T_{s_2} (g, h)} + 4 \sum_{
h\geq0}
\Size{T_{s_1} (g, h) - C\phi^{g-h}} \\
&= 4 \sum_{
h\geq0}
\Size{T_{s_2} (g, h)} + 4 \sum_{-\infty<n\leq g}
\Size{w(n) - C\phi^n}.
\end{align*}
The first sum is $ o\(\phi^{g} \) $ by Lemma \ref{lem:7}. Although the second sum has infinitely many
terms, all those with $ n < 0 $ will contribute a fixed amount since $ w(n) = 0 $ for $ n \leq -2 $.
Thus it suffices to prove that
\begin{equation}
\sum_{
0\leq n\leq g
}\Size{w(n) - C\phi^n} = o\(\phi^{g} \). \label{eq:21}
\end{equation}
For this, we use a method analogous to the proof of \eqref{eq:19}.

Let $ \epsilon > 0 $ be given. By Lemma \ref{lem:10}, there is a $ g_0 $ such that $|w(n) - C\phi^n| < \epsilon\phi^n$ for
all $n \geq g_0$. Then if $g \geq g_0$, then
\begin{align*}
\sum_{
0\leq n\leq g}
\Size{w(n) - C\phi^n} &< \sum_{
0\leq n\leq g_0}
\Size{w(n) - C\phi^n} + \epsilon \sum_{
g_0<n\leq g}
\phi^n \\
&< \sum_{
0\leq n\leq g_0
} \Size{w(n) - C\phi^n} + \epsilon \frac{\phi^{g}}
{1 - \phi}.
\end{align*}
Since the first term is independent of $ g $, we have for all sufficiently large $ g $
\[
\sum_{
  0\leq n\leq g}
\Size{w(n) - C\phi^n} < \epsilon\phi^g
\(
1 + \frac{1}
{1 - \phi}
\)
.
\]
Since $ \epsilon $ was arbitrary, the result \eqref{eq:21} follows. \end{proof}

\section{Conclusions} \label{sec:6}
Our main method of proof has been to reduce questions about $ t (g, h) $, through the
intermediary functions $ T_{\hat s} (g, h) $ and $ \hat s(g, h) $, to questions about $ s(g, h) $. We then exploited known bounds on $ s(g, h) $ to yield, through careful analysis, bounds on $ t (g, h) $.
Thus, further progress on estimating $ t (g, h) $ hinges on improvement of the bounds on
$ s(g, h) $. The same is true of bounds on $N(g) = \sum_{
h} t (g, h)$. Ye’s remarkably crisply
proved result that $N(g + 1) \geq N(g) - N(g - 1)$\footnote{Ye's results \cite{6} have not been published. See Kaplan \cite{new} for an overview.} can be derived from our method,
using no properties of $s(g, h)$ except its nonnegativity; considerably more work is
necessary to prove statements such as Conjecture \ref{conj:1} that do not hold when $s(g, h)$ is
replaced by an arbitrary nonnegative function.

We have mentioned the natural division of the behavior of $s(g, h)$ into “good”
and “bad” regions. Although the “bad” region displays far wilder behavior, even the
“good” region is far from completely understood. In particular, the constant $S$, which
can be computed using only the values $r(k)$ in the “good” region (using Lemmas 10
and 11) has a poorly known value. The best known lower bound is $3.78$, computed by
Zhao by adding terms of a series similar to \eqref{eq:18}; the best known upper bound, found
by examination of Zhai’s delicate argument, is on the order of $10^{16}$. One possible
mode of attacking Conjecture \ref{conj:1} that skirts this problem is to analyze the diagonal
differences $t (g + 1, h + 1) - t (g, h)$ and their analogues for the function $s$. Since $s$
and $t$ are constant on diagonals in the “good” region, we might be able to eliminate
it and focus on sharpening Zhai’s bounds in the “bad” region.

Incidentally, the boundary $g = 2h$ between the “good” and “bad” regions (the
“good” region being where $s_2 = 0$) is apparently quite sharp; it would be interesting
to prove the apparent pattern
\[
  s_2(2h + 2, h) = (-1)^h
\]
and perhaps find simple formulas for $s_2$ on other diagonals parallel to the boundary.

There remains the possibility of looking directly at Bras-Amorós’s more daring
conjecture that $N(g + 2) \geq N(g + 1) + N(g)$ for $g \geq 1$. The difference $N(g + 2) -
N(g + 1) - N(g)$ will, in contrast to Conjecture \ref{conj:1}, grow slower than $\phi^{g}$; computation suggests that it behaves like an exponential $\alpha \psi^{g}$ with $\psi \approx 1.5$. Our methods
provide many ways to transform this difference, possibly leading to a proof that it is
asymptotically positive.
\section*{Acknowledgements}
I would like to thank Joe Gallian and his Research Experience for Undergraduates
(REU) program at the University of Minnesota, Duluth for overseeing this research. The REU was supported by the National Science Foundation (NSF/DMS grant 1062709) and the National Security Agency
(NSA grant H98230-11-1-0224). I thank the Duluth visitors and advisors for their guidance, especially
Eric Riedl, Nathan Kaplan, and Lynnelle Ye. I also thank Alex Zhai for expediently introducing me to
numerical semigroups before the REU began.

I thank Daniel Zhu for pointing out some typos corrected in the present version.

\bibliographystyle{plain}

\textsc{Evan M. O'Dorney, Carnegie Mellon University, Pittsburgh, PA}

\end{document}